\documentclass[a4paper,11pt]{article}
\usepackage{graphicx} 
\usepackage{amsmath}
\usepackage{amssymb}
\usepackage{amsfonts}
\usepackage{amsthm}
\usepackage{parskip}
\usepackage{comment}
\usepackage{xcolor}
\usepackage{tikz}
\usepackage{yhmath}
\usepackage[top=22mm, bottom=23mm, left=22mm, right=22mm]{geometry}
\allowdisplaybreaks

\newtheorem{thm}{Theorem}[section]

\newtheorem{lem}[thm]{Lemma}

\newtheorem{conj}[thm]{Conjecture}
\newtheorem{ques}[thm]{Question}

\theoremstyle{definition}

\title{On a problem of Erd\H{o}s and Ingham}
\author{Fredy Yip\thanks{Trinity College, University of Cambridge, United Kingdom. Email: \textbf{fy276@cam.ac.uk}.}}
\date{}

\begin{document}

\maketitle

\begin{abstract}
    We give a short and elementary argument answering a question of Erd\H{o}s and Ingham negatively. 

    Erd\H{o}s and Ingham showed that a Tauberian estimate they considered was equivalent to the non-vanishing of $1+\sum_{k}a_k^{-1-it}$ for any real number $t$ and any sequence $1<a_1<a_2<\cdots$ of positive integers such that $\sum_k a_k^{-1}<\infty$. We disprove this statement. In fact, we show that for any complex number $\lambda$ and any non-zero real number $t$, there exists a sequence $1<a_1<a_2<\cdots$ of positive integers such that $\sum_k a_k^{-1}<\infty$ and $\sum_k a_k^{-1-it} = \lambda$. 
\end{abstract}

\section{Introduction}

In their 1964 paper on Tauberian estimates, Erd\H{o}s and Ingham~\cite{erdos1964arithmetical} asked for the possibility of zeros of the following zeta-like series on the line $\operatorname{Re} = 1$, where the series is convergent. 

\begin{ques}[Erd\H{o}s and Ingham~\cite{erdos1964arithmetical}] \label{ques}
    Let $1<a_1<a_2<\cdots$ be a (finite or infinite) sequence of integers such that $\sum_k a_k^{-1}<\infty$. Is it true that, for every $t\in \mathbb{R}$,
    \begin{equation*}
        1+\sum_{k}\frac{1}{a_k^{1+it}}\neq 0?
    \end{equation*}
\end{ques}

This problem is listed as problem 967 on Bloom's website~\cite{Bloom2025Erdos967} of problems of Erd\H{o}s. Erd\H{o}s and Ingham were motivated towards Question~\ref{ques} from showing its equivalence to the following Tauberian estimate. 

\begin{thm}[Erd\H{o}s and Ingham~\cite{erdos1964arithmetical}, Theorem 4]
    Let $1<a_1<a_2<\cdots $ be a (finite or infinite) sequence of integers such that $\sum_k a_k^{-1}<\infty$. The following statements are equivalent: 
    \begin{enumerate}
        \item $1+\sum_{k}a_k^{-1-it}\neq 0$ for all $t\in \mathbb{R}$. 
        \item For any non-decreasing $f:\mathbb{R}^{\geq 0}\rightarrow \mathbb{R}^{\geq 0}$ which vanishes on $[0, 1)$, 
        \begin{equation*}
            f(x) + \sum_k f(x/a_k) \sim
            \left(1 + \sum_k a_k^{-1}\right)x
        \end{equation*}
        as $x\rightarrow \infty$ implies that $f(x)\sim x$ as $x\rightarrow \infty$. 
    \end{enumerate}
\end{thm}

Using a short and elementary argument, we give a negative answer to Question~\ref{ques}. In fact, we prove that for any real number $t\neq 0$ and any complex number $\lambda$, there exists a sequence $1 < a_1 < a_2 <\cdots$ of positive integers such that $\sum_k a_k^{-1} < \infty$ and $\sum_k a_k^{-1-it} = \lambda$. 

\begin{thm} \label{main}
    For any real number $t\neq 0$ and any complex number $\lambda$, there exists a subset $S\subseteq \mathbb{Z}^{\geq 2}$ such that
    \begin{align*}
        \sum_{n\in S}\frac{1}{n} &< \infty, \\
        \sum_{n\in S}\frac{1}{n^{1 + it}} &= \lambda. 
    \end{align*}
\end{thm}

It shall follow from our proof of Theorem~\ref{main} that we may in fact demand, for any positive integer $N$ and any real number $\delta > 0$, that $S$ is infinite, $S\subseteq \mathbb{Z}^{\geq N}$ and $\sum_{n\in S}\frac{1}{n} \leq |\lambda| + \delta$. 

\section{Proof}

To show Theorem~\ref{main}, we shall build the desired set $S$ iteratively. In each step, we add a finite number of elements to $S$ to bring $\sum_{n\in S}\frac{1}{n^{1 + it}}$ substantially closer to the target value $\lambda$ without excessively increasing $\sum_{n\in S}\frac{1}{n}$. 

The key to carrying out this strategy is to note that the function $\frac{1}{n^{1 + it}}$ is sufficiently slowly varying. In particular, for small $\epsilon$, the sum $\sum_{n = x}^{(1 + \epsilon)x}\frac{1}{n^{1 + it}}$ has modulus $\Theta(\epsilon)$ and argument $\arg (x) + O(\epsilon)$. We formally capture this observation with the following lemma. 

\begin{lem} \label{lem}
    For any positive integer $N$ and any complex number $c$, there exists a finite set $S'\subseteq \mathbb{Z}^{\geq N}$, such that 
    \begin{align*}
        \left|c - \sum_{n\in S'}\frac{1}{n^{1 + it}}\right|&\leq O\left(|c|^2\right), \\
        \sum_{n\in S'}\frac{1}{n}&\leq |c|, 
    \end{align*}
    where the implied constant (which may be taken to be $1 + |1 + it|$) depends only on $t$. 
\end{lem}

\begin{proof}
    Taking $S' = \emptyset$ suffices for $c = 0$, therefore we assume that $c\neq 0$. Without loss of generality, we may assume that $N\geq \max(|c|^{-2}, |c|^{-1})$. Take $x\geq N$ such that $x^{-it}$ is parallel to $c$. Let $s = \left\lfloor x|c|\right\rfloor\geq 1$. Let $S' = [x, x + s)\cap \mathbb{Z}\subseteq \mathbb{Z}^{\geq N}$ be a set of $s$ integers. We have
    \begin{equation*}
        \sum_{n\in S'}\frac{1}{n}\leq \frac{|S'|}{x} = \frac{s}{x}\leq |c|. 
    \end{equation*}
    By the mean value theorem, we have
    \begin{align*}
        \left|\sum_{n\in S'}\frac{1}{n^{1 + it}} - \frac{|S'|}{x^{1 + it}}\right|\leq |S'|\max_{n\in S'} \left|\frac{1}{n^{1 + it}} - \frac{1}{x^{1 + it}}\right|&\leq |S'|\max_{n\in S'}|n - x|\sup_{y\in [x, x + s]}\left|\frac{1 + it}{y^{2 + it}}\right|\\
        &\leq s^2\frac{|1 + it|}{x^2}\leq |1 + it||c|^2. 
    \end{align*}
    Since $\frac{|S'|}{x^{1 + it}}$ is parallel to $c$, we have
    \begin{equation*}
        \left|\frac{|S'|}{x^{1 + it}} - c\right| = \left|\left|\frac{|S'|}{x^{1 + it}}\right| - |c|\right| = \left|\frac{s}{x} - |c|\right|\leq \frac{1}{x} \leq \frac{1}{N} \leq |c|^2. 
    \end{equation*}
    Therefore, 
    \begin{equation*}
        \left|c - \sum_{n\in S'}\frac{1}{n^{1 + it}}\right|\leq (1 + |1 + it|)|c|^2. 
    \end{equation*}
    as desired. 
\end{proof}

\begin{proof}[Proof of Theorem~\ref{main}]
    We construct $S$ as a union of disjoint finite sets $S_k\subseteq \mathbb{Z}^{\geq 2}$ for $k\in \mathbb{Z}^+$. We construct $S_k$ recursively. Let $r = (2 + 2|1 + it|)^{-1}$. 

    For $k\in \mathbb{Z}^+$, given $S_1, \dots, S_{k - 1}$, take $N_k\geq 2$ greater than all elements of $S_1, \dots,S_{k - 1}$. Let $\lambda_k = \lambda-\sum_{n\in S_1\cup\dots\cup S_{k - 1}}\frac{1}{n^{1 + it}}$. Let $c_k = \lambda_k$ if $|\lambda_k|\leq r$, otherwise, let $c_k$ be parallel to $\lambda_k$ with norm $r$. We construct $S_{k + 1}$ as $S'$ from Lemma~\ref{lem} with $N = N_k$ and $c = c_k$. 

    The finite sets $S_k$ of integers thus obtained are disjoint as any element of $S_k$ is greater than any element of $S_1, \dots, S_{k - 1}$. We take $S = \cup_k S_k$. 

    By Lemma~\ref{lem}, 
    \begin{equation*}
        \left|c_k - \sum_{n\in S_k}\frac{1}{n^{1 + it}}\right|\leq (1 + |1 + it|)|c_k|^2 = \frac{|c_k|^2}{2r}\leq \frac{|c_k|}{2}. 
    \end{equation*}
    Therefore, noting that $|\lambda_k - c_k| = |\lambda_k| - |c_k|$, we have
    \begin{equation*}
        |\lambda_{k + 1}| = \left|\lambda_k - \sum_{n\in S_k}\frac{1}{n^{1 + it}}\right|\leq |\lambda_k - c_k| + \left|c_k - \sum_{n\in S_k}\frac{1}{n^{1 + it}}\right|\leq |\lambda_k| - |c_k|/2. 
    \end{equation*}
    Hence $|\lambda_k|$ decreases as $k$ increases. If $|\lambda_k| > r$, we have
    \begin{equation*}
        |\lambda_{k + 1}| \leq |\lambda_k| - r/2. 
    \end{equation*}
    Therefore, $|\lambda_k|\leq r$ for all $k\geq 2|\lambda|/r$. And for $|\lambda_k|\leq r$, we have
    \begin{equation*}
        |\lambda_{k + 1}| \leq |\lambda_k|/2. 
    \end{equation*}
    Therefore, $\lambda_k\rightarrow 0$ exponentially. By Lemma~\ref{lem}, $\sum_{n\in S_k}\frac{1}{n}\leq |c_k|$. Therefore, 
    \begin{equation*}
        \sum_{n\in S} \frac{1}{n} = \sum_k\sum_{n\in S_k}\frac{1}{n}\leq \sum_k|c_k|\leq \sum_k|r_k| < \infty, 
    \end{equation*}
    and the series $\sum_{n\in S} \frac{1}{n^{1 + it}}$ is absolutely convergent. As such, to show that the series $\sum_{n\in S} \frac{1}{n^{1 + it}}$ converges to $\lambda$, it suffices to consider the subsequence
    \begin{equation*}
        \sum_{n\in S_1\cup \dots \cup S_k}\frac{1}{n^{1 + it}} = \lambda - \lambda_{k + 1}
    \end{equation*}
    of partial sums, which indeed tends to $\lambda$ as $k\rightarrow\infty$. 
\end{proof}

\section{Concluding Remarks}

The elementary approach we consider here is inapplicable to the case where $a_k$ is required to be a finite sequence, where numerous number-theoretical subtleties arise. 

\begin{conj}
    For any finite set $S\subseteq \mathbb{Z}^{\geq 2}$ and any real number $t$, 
    \begin{equation*}
        1 + \sum_{n\in S}\frac{1}{n^{1 + it}}\neq 0. 
    \end{equation*}
\end{conj}

In particular, the special case $S = \{2, 3, 5\}$ considered by Erd\H{o}s and Ingham remains open. 

\begin{ques}[Erd\H{o}s and Ingham~\cite{erdos1964arithmetical}]
    Is it true that, for every $t\in \mathbb{R}$,
    \begin{equation*}
        1 + 2^{-1 - it} + 3^{-1 - it} + 5^{-1 - it}\neq 0?
    \end{equation*}
\end{ques}

\textbf{Acknowledgement.} The author would like to thank Timothy Gowers for helpful suggestions and advice on the presentation of this note. 

\bibliographystyle{abbrv}
\bibliography{mybib}

\begin{thebibliography}{1}

\bibitem{Bloom2025Erdos967}
T.~F. Bloom.
\newblock {Erd\H{o}s Problem \#967}, 2025.
\newblock Accessed: 2025-12-18.

\bibitem{erdos1964arithmetical}
P.~Erd{\"o}s and A.~Ingham.
\newblock Arithmetical tauberian theorems.
\newblock {\em Acta Arithmetica}, 9:341--356, 1964.

\end{thebibliography}

\end{document}